\documentclass[oneside,english]{amsart}
\usepackage[T1]{fontenc}
\usepackage[latin9]{inputenc}
\usepackage{geometry}
\usepackage{amsthm}
\usepackage{amstext}
\usepackage{amssymb}
\usepackage{esint}

\makeatletter
\numberwithin{figure}{section}

\usepackage[T1]{fontenc}
\usepackage[latin9]{inputenc}
\usepackage{geometry}
\usepackage{amsthm}
\usepackage{amstext}
\usepackage{amssymb}
\usepackage{esint}

\makeatletter
\numberwithin{figure}{section}

\usepackage[T1]{fontenc}
\usepackage[latin9]{inputenc}
\usepackage{geometry}
\usepackage{amsthm}
\usepackage{amstext}
\usepackage{amssymb}
\usepackage{esint}

\makeatletter
\numberwithin{figure}{section}

\usepackage[T1]{fontenc}
\usepackage[latin9]{inputenc}
\usepackage{geometry}
\usepackage{amsthm}
\usepackage{amstext}
\usepackage{amssymb}
\usepackage{esint}

\makeatletter
  \theoremstyle{plain}
  \newtheorem{thm}{\protect\theoremname}[section]
  \theoremstyle{definition}
  \newtheorem{defn}{\protect\definitionname}[section]
  \theoremstyle{remark}
  \newtheorem{rem}{\protect\remarkname}[section]
  \theoremstyle{plain}
  \newtheorem{lem}{\protect\lemmaname}[section]
  \theoremstyle{remark}
  \newtheorem*{rem*}{\protect\remarkname}

\makeatother

\usepackage{babel}
  \providecommand{\definitionname}{Definition}
  \providecommand{\lemmaname}{Lemma}
  \providecommand{\remarkname}{Remark}
\providecommand{\theoremname}{Theorem}

\makeatother

\usepackage{babel}

\usepackage{color}

\makeatother

\usepackage{babel}

\makeatother

\usepackage{babel}
\begin{document}

\title{AN ASYMPTOTIC VISCOSITY SELECTION RESULT FOR THE REGULARIZED NEWTON
DYNAMIC}

\author{Boushra Abbas \\
 \\
 I3M UMR CNRS 5149, Université Montpellier II,\\
 34095 Montpellier, France}

\date{April 29, 2015}
\begin{abstract}
Let $\Phi:\mathcal{H}\longrightarrow\mathbb{R\cup}\left\{ +\infty\right\} $
be a closed convex proper function on a real Hilbert space $\mathcal{H}$,
and $\partial\Phi:\mathcal{H}\rightrightarrows\mathcal{H}$ its subdifferential.
For any control function $\epsilon:\mathbb{R}_{+}\longrightarrow\mathbb{R}_{+}$
which tends to zero as $t$ goes to $+\infty$, and $\lambda$ a positive
parameter, we study the asymptotic behavior of the trajectories of
the regularized Newton dynamical system 
\begin{eqnarray*}
 &  & \upsilon\left(t\right)\in\partial\Phi\left(x\left(t\right)\right)\\
 &  & \lambda\dot{x}\left(t\right)+\dot{\upsilon}\left(t\right)+\upsilon\left(t\right)+\varepsilon\left(t\right)x\left(t\right)=0.
\end{eqnarray*}
Assuming that $\varepsilon\left(t\right)$ tends to zero moderately
as $t$ goes to $+\infty$, we show that the term $\varepsilon\left(\cdot\right)x\left(\cdot\right)$
asymptotically acts as a Tikhonov regularization, which forces the
trajectories to converge to a particular equilibrium. Precisely, when
$C=\textrm{argmin}\Phi\neq\emptyset$, and $\varepsilon (\cdot)$ is a ``slow''
control, i.e., $\int_{0}^{+\infty}\varepsilon\left(t\right)dt=+\infty$,
then each trajectory of the system converges weakly, as $t$ goes
to $+\infty$, to the element of minimal norm of the closed convex
set $C.$ When $\Phi$ is a convex differentiable function whose gradient
is Lipschitz continuous,
 we show that the strong convergence property is satisfied. Then we examine the effect of other types of regularizing methods.
\end{abstract}

\maketitle

\section{Introduction }

Throughout this paper, $\mathcal{H}$ is a real Hilbert space with
scalar product $\left\langle \cdot,\cdot\right\rangle $, and $\|x\|^{2}=\left\langle x,x\right\rangle $
for any $x\in\mathcal{H}$. Given $\Phi:\mathcal{H\longrightarrow\mathbb{R}\cup}\left\{ +\infty\right\} $
a closed convex proper function, we will analyze some asymptotic viscosity
selection properties for the regularized Newton dynamic governed by
$\Phi$.

Let us first recall some basic facts about this dynamical system.
Given $\lambda$ a positive constant, the Regularized Newton dynamic
((RN) for short) attached to solving the minimization problem 
\begin{equation*}
(\mathcal P )\quad \underset{x\in\mathcal{H}}{\min}\ \Phi\left(x\right)
\end{equation*}
is written as follows 
\begin{eqnarray}
 &  & \upsilon\left(t\right)\in\partial\Phi\left(x\left(t\right)\right)\label{0}\\
 &  & \lambda\dot{x}\left(t\right)+\dot{\upsilon}\left(t\right)+\upsilon\left(t\right)=0,\label{00}
\end{eqnarray}
where the subdifferential of $\Phi$ at $x\in\textrm{dom}\Phi$ is
classically defined by 
\[
\partial\Phi\left(x\right)=\left\{ p\in\mathcal{H}:\:\Phi\left(y\right)\geq\Phi\left(x\right)+\left\langle p,y-x\right\rangle \;\forall y\in\mathcal{H}\right\} .
\]
When $\Phi$ is a smooth function, (RN) is equivalent to 
\[
\lambda\dot{x}\left(t\right)+\nabla^{2}\Phi(x(t))+\nabla\Phi(x(t))=0
\]
where $\lambda$ acts as a Levenberg-Marquard regularization parameter
of the continuous Newton equation, whence the terminology. This dynamical
system has been first introduced in \cite{AS}, \cite{ARS}. Its extension
to the case of two potentials gives rise to a new class of forward-backward
algorithms, see \cite{AA}, \cite{AAS}, \cite{APR}. In \cite{AS},
for a general closed convex and proper function $\Phi$, it is shown
that the Cauchy problem for the $(x,\upsilon)$ system (\ref{0})-(\ref{00})
admits a unique strong global solution. In addition, under the sole
assumption that $C=\textrm{argmin}\Phi\neq\emptyset$, for any orbit
of (\ref{0})-(\ref{00}), $x\left(t\right)$ converges weakly to
an element of $C$, as $t$ goes to $+\infty$.

In many applications, a particular stationary solution is more interesting
than others due to physical, economic or design considerations. When
we have the global convergence of trajectories, one could let the
trajectory reach a particular target equilibrium by appropriately
adjusting the initial conditions. Nevertheless, in many practical
situations it is not possible to have an accurate control of the initial
state. An alternative approach consists in introducing a term into
the system which forces convergence to the desired stationary solution,
independently of the initial state. Such a term should vanish at
infinity in order to recover, at least asymptotically, an equilibrium
point of (\ref{0})-(\ref{00}).

The above discussion motivates the introduction of the following abstract
evolution system: \begin{subequations} 
\begin{eqnarray}
 &  & \upsilon\left(t\right)\in\partial\Phi\left(x\left(t\right)\right)\label{3-a}\\
 &  & \lambda\dot{x}\left(t\right)+\dot{\upsilon}\left(t\right)+\upsilon\left(t\right)+\varepsilon\left(t\right)x\left(t\right)=0,\label{3-b}
\end{eqnarray}
\end{subequations} where $\epsilon:\mathbb{R}_{+}\longrightarrow\mathbb{R}_{+}$
is an open-loop control function, that tends to zero as $t$ goes
to $+\infty$.

Let us briefly describe our approach. Following a similar device as
in \cite{AS}, setting $\mu=\frac{1}{\lambda}$, and introducing the
new unknown function $y\left(\cdot\right)=x\left(\cdot\right)+\mu\upsilon\left(\cdot\right)$,
we can equivalently rewrite (\ref{3-a})-(\ref{3-b}) as 
\begin{eqnarray*}
 &  & x\left(t\right)=\textrm{prox}_{\mu\Phi}\left(y\left(t\right)\right),\\
 &  & \dot{y}\left(t\right)+\mu\nabla\Phi_{\mu}\left(y\left(t\right)\right)+\mu\epsilon\left(t\right)\textrm{prox}_{\mu\Phi}\left(y\left(t\right)\right)=0
\end{eqnarray*}
where $\textrm{prox}_{\mu\Phi}$ is the proximal mapping associated
to $\mu\Phi$. Recall that $\textrm{prox}_{\mu\Phi}=\left(I+\mu\partial\Phi\right)^{-1}$
is the resolvent of index $\mu>0$ of the maximal monotone operator
$\partial\Phi$, and $\nabla\Phi_{\mu}=\frac{1}{\mu}(I-\left(I+\mu\partial\Phi\right)^{-1})$
is its Yosida approximation of index $\mu>0$. As a key point of our analysis, we notice
that $\textrm{prox}_{\mu\Phi}$ is a gradient vector field, namely
$\textrm{prox}_{\mu\Phi}=\nabla\psi$, with 
\begin{equation}
\psi(y)=\mu\left(\Phi^{*}\right)_{\frac{1}{\mu}}\left(\frac{1}{\mu}y\right),
\end{equation}
where $\Phi^{*}$ is the Fenchel conjugate of $\Phi$. Doing so, we
can reformulate our dynamic in the form \begin{subequations} 
\begin{eqnarray}
 &  & x\left(t\right)=\left(I+\mu\partial\Phi\right)^{-1}\left(y\left(t\right)\right),\label{5-a}\\
 &  & \dot{y}\left(t\right)+\mu\nabla\Phi_{\mu}\left(y\left(t\right)\right)+\epsilon\left(t\right)\nabla\mu\psi\left(y\left(t\right)\right)=0.\label{5-b}
\end{eqnarray}
\end{subequations} Equation (\ref{5-b}) is a particular case of
the multiscale dynamic 
\begin{equation}
\dot{y}\left(t\right)+\partial\Theta(y(t))+\epsilon\left(t\right)\partial\Psi\left(y\left(t\right)\right)\ni0\ \label{gen-visc}
\end{equation}
where $\Theta$ and $\Psi$ are two convex potential functions. Following
\cite{A}, $\Psi$ will be referred to as the \textquotedbl{}viscosity
function\textquotedbl{}. A detailed study of the asymptotic behavior
of the orbits of (\ref{gen-visc}) can be found in \cite{AMO}, \cite{Cabot},
\cite{CPS}, \cite{FurMiyKen}, \cite{Hir}. Following \cite{AMO}
and \cite{Cabot}, we focus our attention on the case where the parametrization
$t\mapsto\epsilon\left(t\right)$ satisfies the following \textquotedbl{}slow\textquotedbl{}
decay property 
\[
\int_{0}^{+\infty}\varepsilon\left(t\right)dt=+\infty.
\]

This condition expresses that $\varepsilon\left(\cdot\right)$ does
not tend to zero too rapidly, which allows the term $\varepsilon\left(\cdot\right)x\left(\cdot\right)$
to be effective asymptotically. In that case, we will show an asymptotic
selection property. Precisely, in Theorem \ref{Theorem 3.2.}, under
some additional moderate growth property on $\epsilon(\cdot)$, we
will show that, for any trajectory $(x,\upsilon)$ of (\ref{3-a})-(\ref{3-b}),
$x(\cdot)$ converges weakly to the minimizer of $\Phi$ which also
minimizes $\psi$ over all minima of $\Phi$. Then we show that this
element is nothing but the element of minimal norm of the solution
set ${\rm \mbox{argmin}}\Phi$, i.e., 
\[
x(t)\rightharpoonup{\rm \mbox{proj}}_{{\rm \mbox{argmin}}\Phi}0\quad\textrm{as}\ t\to+\infty.
\]
Thus we recover the classical Tikhonov viscosity selection principle,
which consists in selecting the solution of minimal norm.

This result can be viewed as an asymptotic selection property: by
using such a slow control $\varepsilon$, one can force all the trajectories
to converge to the same equilibrium, which here is the equilibrium
of minimal norm. This makes a sharp contrast with the non controlled
situation, or fast control, where the limits of the trajectories depend
on the initial data, and are in general difficult to identify.

The paper is organized as follows: we first show the existence and
uniqueness of a strong global solution to the Cauchy problem (\ref{3-a})-(\ref{3-b}).
Then, we study the asymptotic convergence as $t$ goes to $+\infty$
  of the trajectories of (\ref{3-a})-(\ref{3-b}).
In our main result, Theorem \ref{Theorem 3.2.}, under the key assumption
that $\epsilon\left(\cdot\right)$ is a \textquotedbl{}slow control\textquotedbl{}, 
i.e., $\int_{0}^{+\infty}\varepsilon\left(t\right)dt=+\infty$, and
has moderate growth,  we show the weak convergence of the trajectories toward the optimal solution of problem $(\mathcal P)$ of minimum norm. 
When $\Phi$ is a convex differentiable function whose gradient is
Lipschitz continuous, we show that the convergence holds for the strong topology.
Finally, we  examine some variants of this principle
of hierarchical minimization.

\section{Existence and Uniqueness of Global Solutions}

We consider the Cauchy problem for the differential inclusion system
(\ref{3-a})-(\ref{3-b}) \begin{subequations} 
\begin{eqnarray}
 &  & \upsilon\left(t\right)\in\partial\Phi\left(x\left(t\right)\right)\label{2-a}\\
 &  & \lambda\dot{x}\left(t\right)+\dot{\upsilon}\left(t\right)+\upsilon\left(t\right)+\varepsilon\left(t\right)x\left(t\right)=0\label{2-b}\\
 &  & x\left(0\right)=x_{0},\;\upsilon\left(0\right)=\upsilon_{0}\label{2-c}
\end{eqnarray}
\end{subequations} First, we are going to define a notion of strong
solution to the above system. Then, we shall reformulate this system
with the help of the Minty representation of $\partial\Phi$. Finally,
we shall prove the existence and uniqueness of a strong solution to
system (\ref{2-a})\textendash (\ref{2-c}), by applying the Cauchy\textendash Lipschitz
theorem to this equivalent formulation.

\subsection{Definition of strong solutions}

We say that the pair $\left(x\left(\cdot\right),\upsilon\left(\cdot\right)\right)$
is a strong global solution of (\ref{2-a})\textendash (\ref{2-c})
iff the following properties are satisfied:

$i)$ \ $x\left(\cdot\right),\upsilon\left(\cdot\right):\left[0,+\infty\right[\longrightarrow\mathcal{H}$
are absolutely continuous on each interval $\left[0,b\right]$, $0<b<+\infty;$

$ii)$ \ $\upsilon\left(t\right)\in\partial\Phi\left(x\left(t\right)\right)$
\ for all $t\in\left[0,+\infty\right[;$

$iii)$ \ $\lambda\dot{x}\left(t\right)+\dot{\upsilon}\left(t\right)+\upsilon\left(t\right)+\varepsilon\left(t\right)x\left(t\right)=0$
\ for almost all $t\in\left[0,+\infty\right[;$

$iv)$ \ $x\left(0\right)=x_{0},\;\upsilon\left(0\right)=\upsilon_{0}$.

\subsection{Equivalent formulation as a classical differential equation}

In order to solve system (\ref{2-a})\textendash (\ref{2-c}) we use
Minty's device. Set 
\[
\mu=\frac{1}{\lambda}.
\]
Let us rewrite inclusion (\ref{2-a}) by using the following equivalences:
for any $t\in\left[0,+\infty\right[$ 
\begin{eqnarray}
 &  & \upsilon\left(t\right)\in\partial\Phi\left(x\left(t\right)\right)\Leftrightarrow\label{3}\\
 &  & x\left(t\right)+\mu\upsilon\left(t\right)\in x\left(t\right)+\mu\partial\Phi\left(x\left(t\right)\right)\label{4}\\
 &  & x\left(t\right)=\left(I+\mu\partial\Phi\right)^{-1}\left(x\left(t\right)+\mu\upsilon\left(t\right)\right).\label{5}
\end{eqnarray}
Let us introduce the new unknown function $y:\left[0,+\infty\right[\longrightarrow\mathcal{H}$
which is defined for $t\in\left[0,+\infty\right[$ by 
\begin{equation}
y\left(t\right):=x\left(t\right)+\mu\upsilon\left(t\right),\label{6}
\end{equation}
and rewrite the system (\ref{2-a})\textendash (\ref{2-c}) with the
help of $\left(x,y\right)$. From (\ref{5}) and (\ref{6}) 
\begin{eqnarray*}
 &  & x\left(t\right)=\left(I+\mu\partial\Phi\right)^{-1}\left(y\left(t\right)\right),\\
 &  & \upsilon\left(t\right)=\frac{1}{\mu}\left(y\left(t\right)-\left(I+\mu\partial\Phi\right)^{-1}\left(y\left(t\right)\right)\right).
\end{eqnarray*}
Equivalently, 
\begin{eqnarray}
 &  & x\left(t\right)=\textrm{prox}_{\mu\Phi}\left(y\left(t\right)\right);\label{7}\\
 &  & \upsilon\left(t\right)=\nabla\Phi_{\mu}\left(y\left(t\right)\right),\label{8}
\end{eqnarray}
where $\textrm{prox}_{\mu\Phi}$ is the proximal mapping associated
to $\mu\Phi$. Recall that $\textrm{prox}_{\mu\Phi}=\left(I+\mu\partial\Phi\right)^{-1}$
is the resolvent of index $\mu>0$ of the maximal monotone operator
$\partial\Phi$, and $\nabla\Phi_{\mu}$ is its Yosida approximation of
index $\mu>0$.

Let us show how (\ref{2-b}) can be reformulated as a classical differential
equation with respect to $y\left(\cdot\right)$. First, let us rewrite
(\ref{2-b}) as 
\begin{equation}
\dot{x}\left(t\right)+\mu\dot{\upsilon}\left(t\right)+\mu\upsilon\left(t\right)+\mu\varepsilon\left(t\right)x\left(t\right)=0.\label{9}
\end{equation}
Differentiating (\ref{6}), and using (\ref{9}) we obtain 
\begin{eqnarray}
\dot{y}\left(t\right) & = & \dot{x}\left(t\right)+\mu\dot{\upsilon}\left(t\right)\label{10}\\
 & = & -\mu\upsilon\left(t\right)-\mu\varepsilon\left(t\right)x\left(t\right).\label{11}
\end{eqnarray}
From (\ref{7}), (\ref{8}), and (\ref{11}) we deduce that 
\[
\dot{y}\left(t\right)+\mu\nabla\Phi_{\mu}\left(y\left(t\right)\right)+\mu\epsilon\left(t\right)\textrm{prox}_{\mu\Phi}\left(y\left(t\right)\right)=0.
\]
Finally, the $\left(x,y\right)$ system can be written as \begin{subequations}
\begin{eqnarray}
 &  & x\left(t\right)=\textrm{prox}_{\mu\Phi}\left(y\left(t\right)\right)\label{12-a}\\
 &  & \dot{y}\left(t\right)+\mu\nabla\Phi_{\mu}\left(y\left(t\right)\right)+\mu\epsilon\left(t\right)\textrm{prox}_{\mu\Phi}\left(y\left(t\right)\right)=0.\label{12-b}
\end{eqnarray}
\end{subequations} Conversely, if $y\left(\cdot\right)$ is a solution
of (\ref{12-b}), then $\left(x\left(\cdot\right),\upsilon\left(\cdot\right)\right)$
with $x\left(t\right)=\textrm{prox}_{\mu\Phi}\left(y\left(t\right)\right)$
, $\upsilon\left(t\right)=\nabla\Phi_{\mu}\left(y\left(t\right)\right)$
is a solution of (\ref{2-a})-(\ref{2-c}). Let us stress the fact
that the operators $\textrm{prox}_{\mu\Phi}:\mathcal{\mathcal{H}\longrightarrow H}$,
$\nabla\Phi_{\mu}:\mathcal{\mathcal{H}\longrightarrow H}$ are everywhere
defined and Lipschitz continuous, which makes this system relevant
to the Cauchy\textendash Lipschitz theorem.

\subsection{Global existence and uniqueness results}

Let us state our main result of existence and uniqueness for the system
(\ref{2-a})\textendash (\ref{2-c}). \begin{thm} \label{Theorem 2.1.}
Suppose that $\Phi:\mathcal{H}\longrightarrow\mathbb{R}\cup\left\{ +\infty\right\} $
is a convex lower semicontinuous proper function, and $\lambda>0$
is a positive constant. Let $\epsilon:\mathbb{R}_{+}\longrightarrow\mathbb{R}_{+}$
be a nonnegative locally integrable function, and $\left(x_{0},\upsilon_{0}\right)\in\mathcal{\mathcal{H}\times H}$
be such that $\upsilon_{0}\in\partial\Phi\left(x_{0}\right)$. Then
the following properties hold:

i) there exists a unique strong global solution $\left(x\left(\cdot\right),\upsilon\left(\cdot\right)\right):\left[0,+\infty\right[\longrightarrow\mathcal{\mathcal{H}\times H}$
of the Cauchy problem (\ref{2-a})-(\ref{2-c});

ii) the solution pair $\left(x\left(\cdot\right),\upsilon\left(\cdot\right)\right)$
of (\ref{2-a})-(\ref{2-c}) can be represented as follows: for any
$t\in\left[0,+\infty\right[$, 
\begin{eqnarray}
 &  & x\left(t\right)=\textrm{prox}_{\mu\Phi}\left(y\left(t\right)\right);\label{13}\\
 &  & \upsilon\left(t\right)=\nabla\Phi_{\mu}\left(y\left(t\right)\right),\label{14}
\end{eqnarray}
where $y\left(\cdot\right):\left[0,+\infty\right[\longrightarrow\mathcal{H}$
is the unique strong global solution of the Cauchy problem \begin{subequations}
{ 
\begin{eqnarray}
 &  & \dot{y}\left(t\right)+\mu\nabla\Phi_{\mu}\left(y\left(t\right)\right)+\mu\epsilon\left(t\right)\textrm{prox}_{\mu\Phi}\left(y\left(t\right)\right)=0,\label{15-a}\\
 &  & y\left(0\right)=x_{0}+\mu\upsilon_{0}.\label{15-b}
\end{eqnarray}
} 
\end{subequations}
\end{thm} 
\begin{proof} 
Let us first prove the existence and uniqueness of a strong global solution
of the Cauchy problem (\ref{15-a})-(\ref{15-b}). The Cauchy problem
(\ref{15-a})\textendash (\ref{15-b}) can be equivalently written
in abstract form, as the following non-autonomous differential system
\begin{equation}
\ \left\{ \begin{array}{l}
\dot{y}\left(t\right)=F\left(t,y\left(t\right)\right);\\
\rule{0pt}{12pt}y\left(0\right)=x_{0}+\mu\upsilon_{0},
\end{array}\right.\label{abstract-edo}
\end{equation}
with 
\begin{eqnarray}
F\left(t,y\right) & = & G(t,y)+K(t,y),\label{16}\\
G\left(t,y\right) & = & -\mu\nabla\Phi_{\mu}\left(y\right),\label{26-1}\\
K\left(t,y\right) & = & -\mu\epsilon\left(t\right)\textrm{prox}_{\mu\Phi}\left(y\right).\label{26-2}
\end{eqnarray}
In order to apply the Cauchy\textendash Lipschitz theorem to (\ref{abstract-edo}),
let us first examine the Lipschitz continuity properties of $F\left(t,\cdot\right)$.

(a) Take arbitrary $y_{i}\in\mathcal{H}$, $i=1,2$. The Yosida approximation
$\nabla\Phi_{\mu}$ is $\frac{1}{\mu}$-Lipschitz continuous (see
\cite{Br}), and hence, for any $t\geq0,$ $G\left(t,\cdot\right):\mathcal{\mathcal{H}\longrightarrow H}$
is nonexpansive, i.e., 
\begin{equation}
\left\Vert G\left(t,y_{2}\right)-G\left(t,y_{1}\right)\right\Vert \leq\left\Vert y_{2}-y_{1}\right\Vert .\label{17}
\end{equation}
By the nonexpansive property of the resolvent operators we have 
\begin{equation}
\left\Vert K\left(t,y_{2}\right)-K\left(t,y_{1}\right)\right\Vert \leq\mu\epsilon\left(t\right)\left\Vert y_{2}-y_{1}\right\Vert .\label{18}
\end{equation}
Hence, 
\begin{equation}
\left\Vert F\left(t,y_{2}\right)-F\left(t,y_{1}\right)\right\Vert \leq(1+\mu\epsilon\left(t\right))\left\Vert y_{2}-y_{1}\right\Vert .\label{19}
\end{equation}

Since $\varepsilon\left(\cdot\right)$ is nonegative and locally integrable,
(\ref{19}) shows that the Lipschitz constant $L_{F}(t)=(1+\mu\varepsilon\left(t\right))$
of $F(t,\cdot)$ satisfies 
\begin{equation}
L_{F}\left(\cdot\right)\in L^{1}\left(\left[0,b\right]\right)\:\textrm{ for any }0<b<+\infty.\label{19-1}
\end{equation}

(b) Let us show that 
\begin{equation}
\forall y\in\mathcal{H},\:\forall b>0,\;F\left(\cdot,y\right)\in L^{1}\left(\left[0,b\right];\mathcal{H}\right).\label{20}
\end{equation}
Returning to the definition (\ref{16}) of $F$, we deduce that 
\[
\left\Vert F\left(t,y\right)\right\Vert \leq\mu\left\Vert \nabla\Phi_{\mu}\left(y\right)\right\Vert +\mu\epsilon\left(t\right)\left\Vert \textrm{prox}_{\mu\Phi}\left(y\right)\right\Vert .
\]
By assumption, $\epsilon:\mathbb{R}_{+}\longrightarrow\mathbb{R}_{+}$
is a nonnegative locally integrable function, which gives (\ref{20}).
From (\ref{19}) and (\ref{20}), by Cauchy-Lipschitz theorem (see
\cite{Ha}, \cite{Son} for the nonautonomous version used here),
we deduce the existence and uniqueness of a strong global solution
of the Cauchy problem (\ref{abstract-edo}), and hence of (\ref{15-a})-(\ref{15-b}).

\smallskip{}

(2) Let us return to the initial problem (\ref{2-a})-(\ref{2-c}).
Given $y\left(\cdot\right):\left[0,+\infty\right[\longrightarrow\mathcal{H}$
which is the unique strong solution of Cauchy problem (\ref{15-a})-(\ref{15-b}),
let us define $x\left(\cdot\right),\upsilon\left(\cdot\right):\left[0,+\infty\right[\longrightarrow\mathcal{H}$
by 
\begin{equation}
x\left(t\right)=\textrm{prox}_{\mu\Phi}\left(y\left(t\right)\right),\quad\quad\upsilon\left(t\right)=\nabla\Phi_{\mu}\left(y\left(t\right)\right).\label{21}
\end{equation}
(a) Let us show that $x\left(\cdot\right),\upsilon\left(\cdot\right)$
are absolutely continuous on each bounded interval, and satisfy (\ref{2-a})-(\ref{2-c}).
Let us give arbitrary $y_{1}\in\mathcal{H}$, $y_{2}\in\mathcal{H}$.
By the nonexpansive property of the resolvents, we have 
\begin{equation}
\left\Vert \textrm{prox}_{\mu\Phi}\left(y_{2}\right)-\textrm{prox}_{\mu\Phi}\left(y_{1}\right)\right\Vert \leq\left\Vert y_{2}-y_{1}\right\Vert .\label{22}
\end{equation}
Assuming that $s,t\in\left[0,b\right],$ by taking $y_{1}=y\left(s\right)$,
$y_{2}=y(t)$ in (\ref{22}), and owing to the definition of the absolute
continuity property, we deduce that $x\left(t\right)=\textrm{prox}_{\mu\Phi}\left(y\left(t\right)\right)$
is absolutely continuous on $\left[0,b\right]$ for any $b>0$. As
a linear combination of two absolutely continuous functions, the same
property holds true for $\upsilon\left(t\right)=\lambda\left(y\left(t\right)-x\left(t\right)\right).$
\\
 Moreover, for any $t\in\left[0,+\infty\right]$ 
\[
\upsilon\left(t\right)\in\partial\Phi\left(x\left(t\right)\right),\quad\quad y\left(t\right)=x\left(t\right)+\mu\upsilon\left(t\right).
\]
Differentiation of the above equation shows that, for almost every
$t>0$, 
\[
\dot{x}\left(t\right)+\mu\dot{\upsilon}\left(t\right)=\dot{y}\left(t\right).
\]
On the other hand, owing to $\upsilon\left(t\right)=\nabla\Phi_{\mu}\left(y\left(t\right)\right)$,
$x\left(t\right)=\textrm{prox}_{\mu\Phi}\left(y\left(t\right)\right)$,
(\ref{15-a}) can be equivalently written as 
\[
\dot{y}\left(t\right)+\mu\upsilon\left(t\right)+\mu\epsilon\left(t\right)x\left(t\right)=0.
\]
Combining the two above equations, we obtain 
\[
\dot{x}\left(t\right)+\mu\dot{\upsilon}\left(t\right)+\mu\upsilon\left(t\right)+\mu\epsilon\left(t\right)x\left(t\right)=0.
\]
From $\mu=\frac{1}{\lambda}$, we conclude that $\left(x\left(\cdot\right),\upsilon\left(\cdot\right)\right)$
is a solution of system (\ref{2-a})-(\ref{2-b}).\\
 Regarding the initial condition, we observe that 
\begin{eqnarray}
y\left(0\right) & = & x_{0}+\mu\upsilon_{0}\label{23}\\
 & = & x\left(0\right)+\mu\upsilon\left(0\right),\label{24}
\end{eqnarray}
with $\upsilon_{0}\in\partial\Phi\left(x_{0}\right)$ and $\upsilon\left(0\right)\in\partial\Phi\left(x\left(0\right)\right)$.
Hence 
\[
x\left(0\right)=x_{0}=\left(I+\mu\partial\Phi\right)^{-1}\left(x_{0}+\mu\upsilon_{0}\right).
\]
After simplification, we obtain $\upsilon\left(0\right)=\upsilon_{0}$.

(b) Let us now prove the uniqueness. Suppose that 
\[
x\left(\cdot\right),\upsilon\left(\cdot\right):\left[0,+\infty\right[\longrightarrow\mathcal{H}
\]
is a solution pair of (\ref{2-a})-(\ref{2-c}). Defining $\mu=\frac{1}{\lambda}$
and 
\begin{equation}
y\left(t\right)=x\left(t\right)+\mu\upsilon\left(t\right)\label{25}
\end{equation}
we conclude that $y\left(\cdot\right)$ is absolutely continuous,
$y_{0}=x_{0}+\mu\upsilon_{0}$, and for any $t\in\left[0,+\infty\right[$
\begin{equation}
x\left(t\right)=\left(I+\mu\partial\Phi\right)^{-1}\left(y\left(t\right)\right),\quad\upsilon\left(t\right)=\nabla\Phi_{\mu}\left(y\left(t\right)\right).\label{26}
\end{equation}
Since the functions involved in the definition (\ref{25}) of $y\left(\cdot\right)$,
namely $x\left(\cdot\right)$ and $\upsilon\left(\cdot\right)$, are
differentiable for almost all $t\in\left[0,+\infty\right[$, we have
for almost $t\in\left[0,+\infty\right[$ 
\begin{eqnarray*}
\dot{y}\left(t\right) & = & \dot{x}\left(t\right)+\mu\dot{\upsilon}\left(t\right)\\
 & = & -\mu\left(\dot{\upsilon}\left(t\right)+\upsilon\left(t\right)+\epsilon\left(t\right)x\left(t\right)\right)+\mu\dot{\upsilon}\left(t\right).
\end{eqnarray*}
Since $\upsilon\left(t\right)=\nabla\Phi_{\mu}\left(y\left(t\right)\right)$,
we finally obtain 
\[
\dot{y}\left(t\right)+\mu\nabla\Phi_{\mu}\left(y\left(t\right)\right)+\mu\epsilon\left(t\right)\textrm{prox}_{\mu\Phi}\left(y\left(t\right)\right)=0.
\]
Moreover 
\[
y_{0}=x_{0}+\mu\upsilon_{0}.
\]
Arguing as before, by the Cauchy\textendash Lipschitz theorem, the
solution $y\left(\cdot\right)$ of the above system is uniquely determined,
and locally absolutely continuous. Thus, by (\ref{26}), $x\left(\cdot\right)$
and $\upsilon\left(\cdot\right)$ are uniquely determined. \end{proof}

\section{Asymptotic analysis and convergence properties}

In this section, we study the asymptotic behavior, as $t\rightarrow+\infty$,
of the trajectories of system (\ref{2-a})-(\ref{2-b}). Let us recall
our standing assumption, namely the parametrization $\epsilon\left(\cdot\right)$
is supposed to be nonnegative, and locally integrable. In view of
the asymptotic analysis, we also suppose that $\epsilon\left(t\right)\rightarrow0$
as $t\rightarrow\infty$, and satisfies the \textquotedbl{}slow\textquotedbl{}
decay property 
\[
\int_{0}^{+\infty}\varepsilon\left(t\right)dt=+\infty.
\]
By Theorem \ref{Theorem 2.1.}, for any given Cauchy data $\upsilon_{0}\in\partial\Phi\left(x_{0}\right)$,
the above properties guarantee the existence and uniqueness of a global
solution of system (\ref{2-a})-(\ref{2-b})-(\ref{2-c}). From now
on in this section, $\left(x\left(\cdot\right),\upsilon\left(\cdot\right)\right):\left[0,+\infty\right[\longrightarrow\mathcal{\mathcal{H}\times H}$
is the solution of (\ref{2-a})-(\ref{2-b})-(\ref{2-c}). We first
study the asymptotic behavior, as $t\rightarrow+\infty$, of the trajectories
of the associated system (\ref{15-a}) 
\[
\dot{y}\left(t\right)+\mu\nabla\Phi_{\mu}\left(y\left(t\right)\right)+\mu\epsilon\left(t\right)\textrm{prox}_{\mu\Phi}\left(y\left(t\right)\right)=0
\]
whose existence is guaranteed by Theorem \ref{Theorem 2.1.}. The
central point of our analysis is to reformulate this system as a multi-scale
gradient system, which will allow us to use the known results concerning
the asymptotic behavior, and the hierarchical selection property for
such systems.

\subsection{Preliminary results}

Let us state some definitions and classical properties that will be
useful (see \cite{ABM}, \cite{BaCom}, \cite{Br}, \cite{RW} for
an extended presentation of these notions): \begin{defn} \label{Definition 3.1.}Let
$f$ and $g$ be functions from $\mathcal{H}$ to $\mathbb{R}\cup\left\{ +\infty\right\} $.
The \emph{infimal convolution }(or \emph{epi-sum}) of $f$ and $g$
is the function $f\square g:\mathcal{H}\to\left[-\infty,+\infty\right]$
which is defined by 
\[
f\square g(x)=\inf_{\xi\in\mathcal{H}}\left(f\left(\xi\right)+g\left(x-\xi\right)\right).
\]
\end{defn}

\begin{defn} \label{Definition 3.2.} Let $f:\mathcal{H}\rightarrow\mathbb{R}\cup\left\{ +\infty\right\} $,
$\gamma\in\mathbb{R}_{++}$. The Moreau envelope of $f$ of parameter
$\gamma$ is defined by 
\[
f_{\gamma}=f\square\left(\frac{1}{2\gamma}\left\Vert \cdot\right\Vert ^{2}\right).
\]
\end{defn}

\begin{defn} \label{Definition 3.3.} Let $f:\mathcal{H}\rightarrow\mathbb{R}\cup\left\{ +\infty\right\} $
be a convex lower semicontinuous proper function, and let $x\in\mathcal{H}$.
Then $\textrm{prox}_{f}x$ is the unique point in $\mathcal{H}$ that
satisfies 
\[
f_{1}(x)=\min_{\xi\in\mathcal{H}}\left(f\left(\xi\right)+\frac{1}{2}\left\Vert x-\xi\right\Vert ^{2}\right)=f\left(\textrm{prox}_{f}x\right)+\frac{1}{2}\left\Vert x-\textrm{prox}_{f}x\right\Vert ^{2}.
\]
The operator $\textrm{prox}_{f}:\mathcal{H}\rightarrow\mathcal{H}$
is called the proximity operator, or proximal mapping of $f$. \end{defn}

\begin{defn} \label{Definition 3.4.} Let $f:\mathcal{H}\rightarrow\mathbb{R}\cup\left\{ +\infty\right\} $.
The conjugate (or Legendre-Fenchel transform, or Fenchel conjugate)
of $f$ is 
\[
f^{*}:\mathcal{H}\rightarrow\mathbb{R}\cup\left\{ +\infty\right\} :u\mapsto\sup_{x\in\mathcal{H}}\left(\left\langle x,u\right\rangle -f\left(x\right)\right).
\]
\end{defn} \begin{rem} \label{Remark 3.2-1} Let $f:\mathcal{H}\rightarrow\mathbb{R}\cup\left\{ +\infty\right\} $
be proper then 
\[
f^{*}\left(0\right)=-\inf_{\mathcal{H}}f.
\]
$f$ is lower semi continuous and convex if and only if 
\[
f=f^{**}.
\]
\end{rem} \begin{rem} \label{Remark 3.2} Let $f$ and $g$ be functions
from $\mathcal{H}$ to $\mathbb{R}\cup\left\{ +\infty\right\} $.
Then 
\[
\left(f\square g\right)^{*}=f^{*}+g^{*}.
\]
Conversely, if one of the functions ($f$ or $g$) is continuous at
a point of the domain of the other, then 
\[
(f+g)^{*}=f^{*}\square g^{*}.
\]
\end{rem} \begin{rem} \label{Remark 3.3.} a) Let $\varphi:\mathcal{H}\rightarrow\mathbb{R}\cup\left\{ +\infty\right\} $
be proper, and $\gamma\in\mathbb{R}_{++}$. Set $f=\varphi+\frac{1}{2\gamma}\left\Vert \cdot\right\Vert ^{2}$.
Then $\forall u\in\mathcal{H}$ 
\[
f^{*}\left(u\right)=\frac{\gamma}{2}\left\Vert u\right\Vert ^{2}-\varphi_{\gamma}\left(\gamma u\right).
\]
b) Let $C$ be a nonempty subset of $\mathcal{H}$, and let $f=\delta_{C}+\left\Vert \cdot\right\Vert ^{2}/2,$
where $\delta_{C}$ is the indicator function of the set $C$ ($\delta_{C}\left(x\right)=0$
for $x\in\mathcal{H},$ $+\infty$ outwards). Then 
\[
f^{*}=\left(\left\Vert \cdot\right\Vert ^{2}-d_{C}^{2}\right)/2,
\]
where $d_{c}$ is the distance function to the set $C$. For the proof,
set $\varphi=\delta_{C}$ and $\gamma=1$ in Remark \ref{Remark 3.3.}.
\end{rem} 
In the next lemma, we show that the proximal mapping can be written as
the gradient of a convex differentiable function. This result will
play a crucial role in our analysis. \begin{lem} \label{Lemme 3.1}
Let $\Phi:\mathcal{H}\rightarrow\mathbb{R}\cup\left\{ +\infty\right\} $
be a proper convex lower semicontinuous function, and let $\mu>0$.
Then, the proximal mapping $\textrm{prox}_{\mu\Phi}:\mathcal{H}\rightarrow\mathcal{H}$
can be written as the gradient 
\[
\textrm{\ensuremath{\textrm{prox}_{\mu\Phi}}}=\nabla\psi
\]
of the convex continuously differentiable function $\psi:\mathcal{H}\to\mathbb{R}$
which is defined, for any $y\in\mathcal{H}$, by 
\[
\psi(y)=\mu\left(\Phi^{*}\right)_{\frac{1}{\mu}}\left(\frac{1}{\mu}y\right)
\]
where $\Phi^{*}$is the Fenchel conjugate of $\Phi$. \end{lem} \begin{proof}
For any $y\in\mathcal{H}$, set 
\[
x=\textrm{\ensuremath{\textrm{prox}_{\mu\Phi}}}\left(y\right).
\]
By definition of the proximal mapping, we have the following equivalent
formulations 
\begin{eqnarray*}
 &  & x=\left(I+\mu\partial\Phi\right)^{-1}\left(y\right)\\
 &  & y\in\left(I+\mu\partial\Phi\right)\left(x\right)\\
 &  & y-x\in\mu\partial\Phi\left(x\right)\\
 &  & \frac{1}{\mu}\left(y-x\right)\in\partial\Phi\left(x\right).
\end{eqnarray*}
From $\left(\partial\Phi\right)^{-1}=\partial\Phi^{*}$ and the above
equality, we successively obtain 
\begin{align*}
 & x\in\partial\Phi^{*}\left(\frac{1}{\mu}\left(y-x\right)\right);\\
 & \frac{1}{\mu}y\in\frac{1}{\mu}\left(y-x\right)+\frac{1}{\mu}\partial\Phi^{*}\left(\frac{1}{\mu}\left(y-x\right)\right)=\left(I+\frac{1}{\mu}\partial\Phi^{*}\right)\left(\frac{1}{\mu}\left(y-x\right)\right);\\
 & \frac{1}{\mu}\left(y-x\right)=\left(I+\frac{1}{\mu}\partial\Phi^{*}\right)^{-1}\left(\frac{1}{\mu}y\right);\\
 & x=\mu\left(\frac{1}{\mu}y-\left(I+\frac{1}{\mu}\partial\Phi^{*}\right)^{-1}\left(\frac{1}{\mu}y\right)\right);\\
 & x=\mu\left(I-\textrm{\ensuremath{\textrm{prox}_{\frac{1}{\mu}\Phi^{*}}}}\right)\left(\frac{1}{\mu}y\right).
\end{align*}
By the definition of the Yosida approximation of index $\frac{1}{\mu}>0$
of the maximal monotone operator $\partial\Phi^{*}$ 
\[
x=\nabla\left[\left(\Phi^{*}\right)_{\frac{1}{\mu}}\right]\left(\frac{1}{\mu}y\right).
\]
By the classical derivation chain rule, we deduce that the the proximal
mapping $\textrm{prox}_{\mu\Phi}y:\mathcal{H}\rightarrow\mathcal{H}$
is the gradient of the convex continuously differentiable function
$\psi:\mathcal{H}\to\mathbb{R}$ which is defined, for any $y\in\mathcal{H}$,
by 
\[
\psi(y)=\mu\left(\Phi^{*}\right)_{\frac{1}{\mu}}\left(\frac{1}{\mu}y\right).
\]
\end{proof} 
Let us further analyze the function $\psi$, and give equivalent formulations
which come with different proofs of the above lemma. By Definition
\ref{Definition 3.2.} of the Moreau envelope of $\Phi$ of parameter
$\mu$ 
\begin{equation}
\Phi_{\mu}(y)=\inf_{x\in\mathcal{H}}\left\{ \Phi\left(x\right)+\frac{1}{2\mu}\left\Vert y-x\right\Vert ^{2}\right\} ,\label{31}
\end{equation}
and since the Yosida approximation of the subdifferential of $\Phi$
is the Fréchet derivative of Moreau envelope, 
\[
\nabla\Phi_{\mu}\left(y\right)=\frac{1}{\mu}\left(y-J_{\mu}^{\partial\Phi}y\right)
\]
\[
\mu\nabla\Phi_{\mu}\left(y\right)=y-J_{\mu}^{\partial\Phi}y.
\]
By definition, we have $\textrm{\ensuremath{\textrm{prox}_{\mu\Phi}}}\left(y\right)=J_{\mu}^{\partial\Phi}y=\left(I+\mu\partial\Phi\right)^{-1}\left(y\right)$.
Hence 
\[
\mu\nabla\Phi_{\mu}\left(y\right)=y-\textrm{\ensuremath{\textrm{prox}_{\mu\Phi}}}\left(y\right).
\]
\begin{eqnarray*}
\textrm{\ensuremath{\textrm{prox}_{\mu\Phi}}}\left(y\right) & = & y-\mu\nabla\Phi_{\mu}\left(y\right)\\
 & = & \nabla\left(\frac{1}{2}\left\Vert \cdot\right\Vert ^{2}-\mu\Phi_{\mu}\right)\left(y\right).
\end{eqnarray*}
Let us make the link with the previous formulation of the prox as
a gradient, and show that, for any $y\in\mathcal{H}$ 
\[
\frac{1}{2}\left\Vert y\right\Vert ^{2}-\mu\Phi_{\mu}\left(y\right)=\mu\Phi_{\frac{1}{\mu}}^{*}\left(\frac{1}{\mu}y\right);
\]
By (\ref{31}), we have 
\begin{eqnarray*}
\frac{1}{2}\left\Vert y\right\Vert ^{2}-\mu\Phi_{\mu}\left(y\right) & = & \frac{1}{2}\left\Vert y\right\Vert ^{2}-\mu\inf_{x\in\mathcal{H}}\left\{ \Phi\left(x\right)+\frac{1}{2\mu}\left\Vert y-x\right\Vert ^{2}\right\} \\
 & = & \frac{1}{2}\left\Vert y\right\Vert ^{2}-\mu\inf_{x\in\mathcal{H}}\left\{ \Phi\left(x\right)+\frac{1}{2\mu}\left\Vert y\right\Vert ^{2}-\frac{1}{\mu}\left\langle y,x\right\rangle +\frac{1}{2\mu}\left\Vert x\right\Vert ^{2}\right\} \\
 & = & \sup_{x\in\mathcal{H}}\left\{ \left\langle y,x\right\rangle -\mu\left(\Phi\left(x\right)+\frac{1}{2\mu}\left\Vert x\right\Vert ^{2}\right)\right\} \\
 & = & \mu\sup_{x\in\mathcal{H}}\left\{ \left\langle \frac{1}{\mu}y,x\right\rangle -\left(\Phi\left(x\right)+\frac{1}{2\mu}\left\Vert x\right\Vert ^{2}\right)\right\} .
\end{eqnarray*}
By using Remark \ref{Remark 3.2} concerning the conjugate of a sum,
we obtain %
\begin{eqnarray*}
\frac{1}{2}\left\Vert y\right\Vert ^{2}-\mu\Phi_{\mu}\left(y\right) & = & \mu\left(\Phi+\frac{1}{2\mu}\left\Vert \cdot\right\Vert ^{2}\right)^{*}\left(\frac{1}{\mu}y\right)\\
 & = & \mu\left(\Phi^{*}\square\frac{\mu}{2}\left\Vert \cdot\right\Vert ^{2}\right)\left(\frac{1}{\mu}y\right)\\
 & = & \mu\left(\Phi^{*}\right)_{\frac{1}{\mu}}\left(\frac{1}{\mu}y\right).
\end{eqnarray*}
So we obtain the same function $\psi$ as given in Lemma \ref{Lemme 3.1}.
Note that by Remark \ref{Remark 3.3.}, the equivalent (dual) formulation
of $\psi$ given by $\psi(y)=\frac{1}{2}\left\Vert y\right\Vert ^{2}-\mu\Phi_{\mu}\left(y\right)$,
which is written as a d.c. function, is actually a convex function.

\subsection{Asymptotic hierarchical minimization}

Let us study the asymptotic behavior of the trajectories of system
(\ref{2-a})-(\ref{2-b}). We consider the equivalent system (\ref{15-a}),
which, by Lemma \ref{Lemme 3.1}, can be formulated as follows: 
\begin{eqnarray}
 &  & x\left(t\right)=\textrm{prox}_{\mu\Phi}\left(y\left(t\right)\right)\label{33}\\
 &  & \dot{y}\left(t\right)+\nabla\Theta\left(y\left(t\right)\right)+\epsilon\left(t\right)\nabla\Psi\left(y\left(t\right)\right)=0,\label{34}
\end{eqnarray}
where, for any $y\in\mathcal{H}$ 
\begin{eqnarray}
 &  & \Theta(y):=\mu\Phi_{\mu}(y);\label{33-a}\\
 &  & \Psi\left(y\right)=\mu^{2}\left(\Phi^{*}\right)_{\frac{1}{\mu}}\left(\frac{1}{\mu}y\right).\label{34-a}
\end{eqnarray}
Note that $\Theta$ and $\Psi$ are two convex continuously differentiable
functions. We are within the framework of the multiscale gradient
system ${\rm (MAG)_{\varepsilon}}$, with a positive control $t\mapsto\varepsilon(t)$
that converges to $0$ as $t\to\infty$, 
\[
{\rm (MAG)_{\varepsilon}}\qquad\dot{y}(t)+\partial\Theta(y(t))+\varepsilon(t)\partial\Psi(y(t))\ni0,
\]
that has been considered by Attouch-Czarnecki in \cite{AMO}. Let
us recall this general abstract result, that we formulate with notations
adapted to our setting. Since $\Theta$ enters ${\rm (MAG)_{\varepsilon}}$
only by its subdifferential, it is not a restrictive assumption to
assume this potential to be nonnegative, with its infimal value equal
to zero (substracting the infimal value does not affect the subdifferential).

\begin{thm}{(Attouch-Czarnecki, \cite{AMO})} \label{Theweak_eps}
Let 
\begin{itemize}
\item $\Theta:\mathcal{H}\rightarrow\mathbb{R}^{+}\cup\{+\infty\}$ be a
closed convex proper function, such that $C$ = $\mbox{{\rm argmin}}\Theta=\Theta^{-1}(0)\neq\emptyset.$

\smallskip{}

\item $\Psi:\mathcal{H}\rightarrow\mathbb{R}^{\phantom{+}}\cup\{+\infty\}$
be a closed convex proper function, such that $S=\mbox{{\rm argmin}}\{\Psi|\mbox{{\rm argmin}}\Theta\}\neq\emptyset.$ 
\end{itemize}
Let us assume that, 
\begin{itemize}
\item ${(\mathcal{H}_{1})}_{\varepsilon}$%
\mbox{%
}%
\mbox{%
} $\forall p\in R(N_{C}),$ ${\displaystyle \int_{0}^{+\infty}\Theta^{*}(\varepsilon(t)p)-\sigma_{C}(\varepsilon(t)p)dt<+\infty.}$ 
\item ${(\mathcal{H}_{2})}_{\varepsilon}$%
\mbox{%
}%
\mbox{%
} $\varepsilon(\cdot)$ is a non increasing function of class $C^{1}$,
such that $\lim_{t\to+\infty}\varepsilon(t)=0$, ${\displaystyle \int_{0}^{+\infty}\varepsilon(t)dt=+\infty}$,
and for some $k\geq0$, $-k\varepsilon^{2}\leq\dot{\varepsilon}$. 
\end{itemize}
Let $y(\cdot)$ be a strong solution of ${\rm (MAG)_{\varepsilon}}$.
Then: 
\begin{eqnarray*}
(i) & \mbox{ weak convergence } & \exists y_{\infty}\in S=\mbox{{\rm {\rm argmin}}}\{\Psi|\mbox{{\rm argmin}}\Theta\},\qquad w-\lim_{t\to+\infty}y(t)=y_{\infty};\\
(ii) & \mbox{ minimizing properties } & \lim_{t\to+\infty}\Theta(y(t))=0;\\
 &  & \lim_{t\to+\infty}\Psi(y(t))=\min\Psi|_{\mbox{{\rm argmin}}{\Theta}};\\
(iii) &  & \forall z\in S\lim_{t\to+\infty}\|y(t)-z\|\mbox{ exists };\\
(iv) & \mbox{ estimations } & \lim_{t\to+\infty}\frac{1}{\varepsilon(t)}\Theta(y(t))=0;\\
 &  & \int_{0}^{+\infty}\Theta(y(t))dt<+\infty;\\
 &  & \limsup_{\tau\to+\infty}\int_{0}^{\tau}\varepsilon(t)\left(\Psi(y(t))-\min\Psi|_{\mbox{{\rm argmin}}\Theta}\right)dt<+\infty.
\end{eqnarray*}
\end{thm} By specializing this result to our setting, we will obtain
the weak convergence of $y(\cdot)$ to a particular minimizer of $\Phi$,
which is the solution of a hierarchical minimization property. The
convergence of $x(\cdot)$ is less immediate, and will follow from
an energetical argument.

\subsection*{Analysis of the condition $(\mathcal{H}_{1})_{\epsilon}$:}

The condition 

\begin{center}
${(\mathcal{H}_{1})}_{\varepsilon}$ $\quad$ $\forall p\in R(N_{C}),$
$\quad$ ${\displaystyle \int_{0}^{+\infty}\Theta^{*}(\varepsilon(t)p)-\sigma_{C}(\varepsilon(t)p)dt<+\infty,}$
$\vphantom{}$ 
\par\end{center}

plays a crucial role in our asymptotic analysis. Before proceeding
in the discussion of this hypothesis, we recall some classical notions
from convex analysis, that will be useful. 
\begin{itemize}
\item $\sigma_{C}$ is the support function of $C$, 
\[
\sigma_{C}\left(x^{*}\right)=\underset{c\in C}{\sup}\left\langle x^{*},c\right\rangle .
\]

\item $N_{C}\left(x\right)$ is the normal cone to $C$ at $x,$ 
\[
N_{C}\left(x\right)=\left\{ x^{*}\in\mathcal{H}:\,\left\langle x^{*},c-x\right\rangle \leq0\textrm{ for all }c\in C\right\} \mbox{ if }x\in C,\mbox{ and }\emptyset\mbox{ otherwise}.
\]

\smallskip{}

\item $R\left(N_{C}\right)$ is the range of $N_{C}$, i.e. $p\in R\left(N_{C}\right)$
if and only if $p\in N_{C}\left(x\right)$ for some $x\in C.$

\smallskip{}

\item Note that $\delta_{C}^{*}=\sigma_{C}$ where $\delta_{C}$ is the
indicator function of $C$, 
\[
\delta_{C}:=\begin{cases}
0 & \textrm{if }x\in C\\
+\infty & \textrm{otherwise.}
\end{cases}
\]

\end{itemize}
$\vphantom{}$

$\vphantom{}$

{\if First, we show that the system ( where $\epsilon\left(t\right)\longrightarrow0$
as $t\longrightarrow+\infty$)

$\vphantom{}$

$\vphantom{}$

$\left(\textrm{MAG}\right)_{\epsilon}$$\qquad\qquad\qquad\qquad\qquad$$\dot{y}\left(t\right)+\partial\Theta\left(y\left(t\right)\right)+\epsilon\left(t\right)\partial\Psi\left(y\left(t\right)\right)=0,$

$\vphantom{}$

after time rescaling, $\left(\textrm{MAG}\right)_{\epsilon}$ can
be equivalently written as:

$\vphantom{}$

$\left(\textrm{MAG}\right)_{\beta}$$\qquad\qquad\qquad\qquad\qquad$$\dot{y}\left(t\right)+\partial\Psi\left(y\left(t\right)\right)+\beta\left(t\right)\partial\Theta\left(y\left(t\right)\right)=0$

$\vphantom{}$

with a positive control $t\longrightarrow\beta\left(t\right),$ $\beta\left(t\right)\longrightarrow+\infty$
as $t\longrightarrow+\infty.$

$\vphantom{}$

By \cite[Lemma 4.1.]{AMO}, there is an equivalence between the formulation
$\left(\textrm{MAG}\right)_{\epsilon}$ and $\left(\textrm{MAG}\right)_{\beta}$.

Let us observe that $\Psi+\beta\left(t\right)\Theta\uparrow\Psi+\delta_{C}$
as $t\longrightarrow+\infty$, where $\delta_{C}$ is the indicator
function of the set $C=\underset{C}{\arg\min}\Theta$ ($\delta_{C}\left(x\right)=0$
for $x\in C$, $+\infty$ outwards).

\fi}

Observe that in ${(\mathcal{H}_{1})}_{\varepsilon}$, all the terms
in the integral are nonnegative. Indeed, since $\Theta$ is bounded
from above by the indicator function of the set $C$, i.e. $\Theta\leq\delta_{C}$
(recall that $\Theta=0$ on $C$), the reverse inequality holds for
their Fenchel conjugates, whence 
\[
\Theta^{*}\left(\epsilon\left(t\right)p\right)-\sigma_{C}\left(\epsilon\left(t\right)p\right)\geq0\qquad\forall p\in\mathcal{H}.
\]
Thus, Hypothesis $\left(\mathcal{H}_{1}\right)_{\epsilon}$ means
that, for all $p\in R(N_{C})$ the nonnegative function 
\[
t\mapsto\left[\Theta^{*}\left(\epsilon\left(t\right)p\right)-\sigma_{C}\left(\epsilon\left(t\right)p\right)\right]
\]
is integrable on $\left(0,+\infty\right)$. For more clarity, let
us discuss the following special case: Suppose that 
\[
\Theta\left(x\right)\geq\frac{r}{2}\textrm{dis}^{2}\left(x,C\right),
\]
for some $r>0$. Then $\Theta^{*}\left(x\right)\leq\frac{1}{2r}\left\Vert x\right\Vert ^{2}+\sigma_{C}\left(x\right)$
and 
\[
\Theta^{*}\left(z\right)-\sigma_{C}\left(z\right)\leq\frac{1}{2r}\left\Vert z\right\Vert ^{2} .
\]
Hence, in this situation $\left(\mathcal{H}_{1}\right)_{\epsilon}$
is satisfied if the following condition on $\epsilon(\cdot)$ is satisfied:
\[
\int_{0}^{+\infty}\epsilon^{2}\left(t\right)<+\infty.
\]
$\vphantom{}$ In this situation, the moderate growth condition on
$\epsilon(\cdot)$, can be formulated as 
\[
\epsilon(\cdot)\in L^{2}(0,+\infty)\smallsetminus L^{1}(0,+\infty).
\]
Let us return to the general situation, and summarize our results
in the following theorem, which is our main statement.


\begin{thm}\label{Theorem 3.2.} Let $\Phi:\mathcal{H}\rightarrow\mathbb{R}^{+}\cup\{+\infty\}$
be a closed convex function, such that $C$ = $\mbox{{\rm argmin}}\Phi = \Phi^{-1}(0) \neq\emptyset.$
Let us assume that, 
\begin{itemize}
\item ${(\mathcal{H}_{1})}_{\varepsilon}$%
\mbox{%
}%
\mbox{%
} $\forall p\in R(N_{C}),$ ${\displaystyle \int_{0}^{+\infty}\mu\Phi^{*}(\frac{1}{\mu}\varepsilon(t)p)-\sigma_{C}(\varepsilon(t)p)+\frac{1}{2}\|\varepsilon(t)p\|^{2}dt<+\infty;}$ 
\item ${(\mathcal{H}_{2})}_{\varepsilon}$%
\mbox{%
}%
\mbox{%
} $\varepsilon(\cdot)$ is a nonincreasing function of class $C^{1}$,
Lipschitz continuous on $[0,+\infty[$, and such that $\lim_{t\to+\infty}\varepsilon(t)=0$,
${\displaystyle \int_{0}^{+\infty}\varepsilon(t)dt=+\infty}$, and
for some $k\geq0$, $-k\varepsilon^{2}\leq\dot{\varepsilon}$. 
\end{itemize}
Then, for any trajectory $\left(x\left(\cdot\right),\upsilon\left(\cdot\right)\right):\left[0,+\infty\right[\longrightarrow\mathcal{\mathcal{H}\times H}$
solution of (\ref{2-a})-(\ref{2-b}), with $y\left(t\right)=x\left(t\right)+\mu\upsilon\left(t\right)$

\medskip{}

$\ \quad\quad\ \ (i)\mbox{ weak convergence }\quad w-\lim_{t\to+\infty}y(t)=\mbox{{\rm proj}}_{\mbox{{\rm argmin}}\Phi}0.$

\medskip{}

Let us further assume that \ $\Phi(0)<+\infty$. Then 
\begin{eqnarray*}
(ii)\  & \mbox{ weak convergence } & \qquad w-\lim_{t\to+\infty}x(t)=w-\lim_{t\to+\infty}y(t)=\mbox{{\rm proj}}_{\mbox{{\rm argmin}}\Phi}0;\\
(iii) & \mbox{ strong convergence } & s-\lim_{t\to+\infty}v(t)=0,\quad and\ hence\ s-\lim_{t\to+\infty}x(t)-y(t)=0.
\end{eqnarray*}
\end{thm} \begin{proof} Let us apply Theorem \ref{Theweak_eps}
with $\Theta(y)=\mu\Phi_{\mu}(y)$, and $\Psi\left(y\right)=\mu^{2}\left(\Phi^{*}\right)_{\frac{1}{\mu}}\left(\frac{1}{\mu}y\right)$,
which are two convex continuously differentiable functions. By the
general properties of the Moreau enveloppe, the function $\Theta$
is still nonnegative, and 
\[
\mbox{{\rm argmin}}\Theta=\mbox{{\rm argmin}}\Phi_{\mu}=\mbox{{\rm argmin}}\Phi=C.
\]
Moreover, on $C$ we have $\Theta =0$. Let us particularize the conditions ${(\mathcal{H}_{1})}_{\varepsilon}$
and ${(\mathcal{H}_{2})}_{\varepsilon}$ to our setting. Let us first
compute $\Theta^{*}$, the conjugate of $\Theta$. For any $z\in\mathcal{H}$
\begin{align*}
\Theta^{*}(z) & =\sup_{y}\left\lbrace \langle z,y\rangle-\mu\Phi_{\mu}(y)\right\rbrace \\
 & =\mu\sup_{y}\left\lbrace \langle\frac{1}{\mu}z,y\rangle-\Phi_{\mu}(y)\right\rbrace \\
 & =\mu\left(\Phi_{\mu}\right)^{*}\left(\frac{1}{\mu}z\right)\\
 & =\mu\left(\Phi^{*}\left(\frac{1}{\mu}z\right)+\frac{\mu}{2}\|\frac{1}{\mu}z\|^{2}\right)\\
 & =\mu\Phi^{*}\left(\frac{1}{\mu}z\right)+\frac{1}{2}\|z\|^{2}.
\end{align*}
Hence conditions ${(\mathcal{H}_{1})}_{\varepsilon}$ and ${(\mathcal{H}_{2})}_{\varepsilon}$
of Theorem \ref{Theweak_eps} are satisfied. As a consequence, we
obtain the convergence of $y(\cdot)$ to a solution $y_{\infty}$
of the constrained minimization problem 
\begin{equation}
\min\left\lbrace \Psi(x):\ x\in C\right\rbrace .\label{visco-select1}
\end{equation}

Let us write the first-order optimality condition satisfied by $y_{\infty}$.
We have 
\begin{equation}
\nabla\Psi(y_{\infty})+N_{C}(y_{\infty})\ni0.\label{visco-select2}
\end{equation}
Since $\nabla\Psi=\mu\textrm{\ensuremath{\textrm{prox}_{\mu\Phi}}}$,
equivalently 
\begin{equation}
\mu\textrm{\ensuremath{\textrm{prox}_{\mu\Phi}}}(y_{\infty})+N_{C}(y_{\infty})\ni0.\label{visco-select3}
\end{equation}
Noticing that $y_{\infty}\in C$, and that $z=\textrm{\ensuremath{\textrm{prox}_{\mu\Phi}}}z$
for $z\in C=\mbox{{\rm argmin}}\Phi$, we obtain 
\begin{equation}
\mu y_{\infty}+N_{C}(y_{\infty})\ni0.\label{visco-select4}
\end{equation}
By definition of $N_{C}$, equivalently, the following property is
satisfied 
\[
\left\langle 0-y_{\infty},c-y_{\infty}\right\rangle \leq0\quad\forall c\in C.
\]
Since $y_{\infty}\in C$, this is the condition of the obtuse angle
that characterizes the projection of the origin on $C$. Thus 
\begin{equation}
y_{\infty}={\rm \mbox{proj}}_{C}(0).\label{visco-select5}
\end{equation}
We have obtained that $y(\cdot)$ converges weakly to the element
of minimal norm of the solution set $C$, that's item $(i)$. In order
to pass from the convergence of $y(\cdot)$ to the convergence of
$x(\cdot)$ we use the relation (\ref{7}) 
\[
x\left(t\right)=\textrm{prox}_{\mu\Phi}\left(y\left(t\right)\right)
\]
that links the two variables.\\
 In a finite dimensional setting, we can conclude the strong convergence of
$x(\cdot)$  thanks to the continuity
of the proximal mapping, and using again the fact that the set $C$
of minimizers of $\Phi$ is invariant by the proximal mapping $\textrm{prox}_{\mu\Phi}$,
i.e., 
\[
\textrm{prox}_{\mu\Phi}(y)=y\quad{\rm for\ all\ }y\in C=\mbox{{\rm argmin}}\Phi.
\]
In an infinite dimensional setting, we are going to use the particular
structure of our dynamical system, and an energetical argument to
show that 
\begin{equation}
x(t)-y(t)\to0\mbox{ strongly as }t\to+\infty.\label{energy1}
\end{equation}
This will result from the finite energy property 
\begin{equation}
\int_{0}^{\infty}\|\dot{y}(t)\|^{2}dt<+\infty.\label{energy2}
\end{equation}
To obtain (\ref{energy2}), take the scalar product of equation (\ref{34})
with $\dot{y}\left(t\right)$. We obtain 
\begin{equation}
\|\dot{y}\left(t\right)\|^{2}+\frac{d}{dt}\left(\Theta\left(y\left(t\right)\right)\right)+\epsilon\left(t\right)\frac{d}{dt}\left(\Psi\left(y\left(t\right)\right)\right)=0.
\end{equation}
After integration by parts we obtain, for any $T>0$ 
\begin{equation}
\int_{0}^{T}\|\dot{y}\left(t\right)\|^{2}dt+\Theta\left(y\left(T\right)\right)-\Theta\left(y\left(0\right)\right)+\epsilon\left(T\right)\Psi\left(y\left(T\right)\right)-\epsilon\left(0\right)\Psi\left(y\left(0\right)\right)-\int_{0}^{T}\dot{\epsilon}(t)\Psi\left(y\left(t\right)\right)dt=0.\label{energy3}
\end{equation}
By assumption $\Phi(0)<+\infty$. By Remark \ref{Remark 3.2-1} $\Phi^{**}=\Phi$,
we equivalently have $\inf\Phi^{*}>-\infty$, and hence 
\[
\inf\Psi=\inf\mu^{2}\left(\Phi^{*}\right)_{\frac{1}{\mu}}=\mu^{2}\inf\Phi^{*}>-\infty.
\]
Set $m:=\inf\Psi$. From (\ref{energy3}), and $\dot{\epsilon}(t)\leq0$
(recall that $\varepsilon(\cdot)$ is a nonincreasing function) we
deduce that 
\begin{equation}
\int_{0}^{T}\|\dot{y}\left(t\right)\|^{2}dt\leq\Theta\left(y\left(0\right)\right)+\epsilon\left(0\right)\Psi\left(y\left(0\right)\right)+|m|\epsilon\left(T\right)+m\int_{0}^{T}\dot{\epsilon}(t)dt.\label{energy4}
\end{equation}
Since the above majorization is valid for any $T>0$, and $\epsilon$
is bounded (it decreases to zero), we obtain (\ref{energy2}).\\
 We now observe that $y(\cdot)$ is Lipschitz continuous on $[0,+\infty[$.
This follows from equation (\ref{34}), and the following argument.
Since $y(\cdot)$ is converging weakly, it is bounded. By the Lipschitz
continuity of the operators $\nabla\Theta=\mu\nabla\Phi_{\mu}$ and
$\nabla\Psi=\mu\textrm{prox}_{\mu\Phi}$ (which are therefore bounded
on bounded sets), and by equation (\ref{34}), we deduce that $\dot{y}(\cdot)$
is bounded, and hence $y(\cdot)$ is Lipschitz continuous on $[0,+\infty[$.\\
 Using again equation (\ref{34}), and the Lipschitz continuity properties
of $\nabla\Theta$ and $\nabla\Psi$, $y$, and $\varepsilon$ (for
this last property note that for some $c>0$, $-c\leq-k\varepsilon^{2}\leq\dot{\varepsilon}\leq0$),
we deduce that $t\mapsto\dot{y}(t)$ is Lipschitz continuous on $[0,+\infty[$.
Hence $\dot{y}(\cdot)$ belongs to $L^{2}([0,+\infty[;\mathcal{H})$,
and is Lipschitz continuous. By a classical result this implies 
\[
\lim_{t\to+\infty}\dot{y}(t)=0.
\]
Returning to (\ref{34}), and noticing that $\varepsilon(t)\nabla\Psi(y(t))\to0$,
we obtain 
\[
\lim_{t\to+\infty}\nabla\Theta(y(t))=0.
\]
Since $\nabla\Theta(y(t))=y(t)-\textrm{\ensuremath{\textrm{prox}_{\mu\Phi}}}\left(y(t)\right)=y(t)-x(t)$,
we finally obtain that $x(t)-y(t)$ converges strongly to zero as
$t\longrightarrow+\infty$, which clearly implies that $x(\cdot)$
and $y(\cdot)$ converge weakly to the same limit, which is the solution
of a hierarchical minimization problem. \end{proof}

\noindent \textbf{Example}: Let us return to our model situation where
\[
\Phi\left(x\right)\geq\frac{r}{2}\textrm{dis}^{2}\left(x,C\right),
\]
for some $r>0$. Then $\Phi^{*}\left(x\right)\leq\frac{1}{2r}\left\Vert x\right\Vert ^{2}+\sigma_{C}\left(x\right)$
and 
\[
\Phi^{*}\left(z\right)-\sigma_{C}\left(z\right)\leq\frac{1}{2r}\left\Vert z\right\Vert ^{2}.
\]
After elementary computation, one can verify that, in this situation,
$\left(\mathcal{H}_{1}\right)_{\epsilon}$ is satisfied if the following
condition on $\epsilon(\cdot)$ is satisfied: 
\[
\int_{0}^{+\infty}\epsilon^{2}\left(t\right)<+\infty.
\]
$\vphantom{}$ Thus, in this situation, the moderate growth condition
on $\epsilon(\cdot)$, can be formulated as 
\[
\epsilon(\cdot)\in L^{2}(0,+\infty)\smallsetminus L^{1}(0,+\infty).
\]

\subsection{Strong convergence}

Let us now examine the strong convergence properties of the trajectories.
Let us first consider the variable $y(\cdot)$. Following \cite[Theorem 2.2]{AMO},
and equation (\ref{34}), the strong convergence of $y(\cdot)$ will
result from the strong monotonicity property of $\nabla\Psi=\mu\textrm{\ensuremath{\textrm{prox}_{\mu\Phi}}}$.
We recall that $\nabla\Psi$ is said to be strongly monotone if there
exists some $\alpha>0$ such that for any $x\in\textrm{dom}\nabla\Psi$,
and $y\in\textrm{dom}\nabla\Psi$ 
\[
\left\langle \nabla\Psi\left(x\right)-\nabla\Psi\left(y\right),x-y\right\rangle \geq\alpha\left\Vert x-y\right\Vert ^{2}.
\]

This property turns out to be equivalent to a regularity property
for $\Phi$, as stated in the following Lemma. \begin{lem}\label{strong1}
Let $\Phi$ be a convex differentiable function whose gradient is
L-Lipschitz continuous for some $L>0$. Then, for any $\mu>0$ such
that $\mu L<1$, the proximal mapping $\textrm{\ensuremath{\textrm{prox}_{\mu\Phi}}}$
is strongly monotone. \end{lem} \begin{proof} Take $y_{i}$, $i=1,2$.
By definition of $\textrm{\ensuremath{\textrm{prox}_{\mu\Phi}}}(y_{i})$,
\ $\textrm{\ensuremath{\textrm{prox}_{\mu\Phi}}}(y_{i})+\mu\nabla\Phi(\textrm{\ensuremath{\textrm{prox}_{\mu\Phi}}}(y_{i}))=y_{i}$.
Taking the difference of the two equations, and multiplying scalarly
by $y_{2}-y_{1}$, we obtain 
\[
\left\langle \textrm{\ensuremath{\textrm{prox}_{\mu\Phi}}}(y_{2})-\textrm{\ensuremath{\textrm{prox}_{\mu\Phi}}}(y_{1}),y_{2}-y_{1}\right\rangle +\mu\left\langle \nabla\Phi(\textrm{\ensuremath{\textrm{prox}_{\mu\Phi}}}(y_{2})-\nabla\Phi(\textrm{\ensuremath{\textrm{prox}_{\mu\Phi}}}(y_{1}),y_{2}-y_{1}\right\rangle =\|y_{2}-y_{1}\|^{2}.
\]
Then use the Cauchy-Schwarz inequality, the $L$-Lipschitz continuity
of $\nabla\Phi$, and the fact that the proximal mapping is nonexpansive
to obtain 
\[
\left\langle \textrm{\ensuremath{\textrm{prox}_{\mu\Phi}}}(y_{2})-\textrm{\ensuremath{\textrm{prox}_{\mu\Phi}}}(y_{1}),y_{2}-y_{1}\right\rangle \geq(1-\mu L)\|y_{2}-y_{1}\|^{2}.
\]
Conversely, one can easily establish that the strong monotonicity
of the proximal mapping implies that $\Phi$ is a convex differentiable
function whose gradient is Lipschitz continuous. \end{proof}

We can now complete Theorem \ref{Theorem 3.2.} as follows. \begin{thm}\label{strong2}
Let us make the assumptions of Theorem \ref{Theorem 3.2.}, and assume
moreover that $\Phi$ is a convex differentiable function whose gradient
is L-Lipschitz continuous for some $L>0$. Then for $\mu L<1$, we
have the strong convergence property of $x(\cdot)$ and $y(\cdot)$
to the element of minimal norm of $C$ = $\mbox{{\rm argmin}}\Phi\neq\emptyset.$
\end{thm} 
\begin{proof} By Theorem \ref{Theorem 3.2.} item $(iii)$,
$x(t)-y(t)$ converges strongly to zero as $t\longrightarrow+\infty$.
Hence we just need to prove that $y(\cdot)$ converges strongly .
Since $\mu L<1$, by Lemma \ref{strong1}, the operator $\nabla\Psi=\mu\textrm{\ensuremath{\textrm{prox}_{\mu\Phi}}}$
is strongly monotone. Thus we are in the situation examined in \cite[Theorem 2.2]{AMO},
which gives the strong convergence property. Another equivalent approach
consists in noticing that by Theorem \ref{Theweak_eps}, we have $w-\lim_{t\to+\infty}y(t)=\mbox{{\rm proj}}_{\mbox{{\rm argmin}}\Phi}0$
and $\Psi(y(t))\to\Psi(\mbox{{\rm proj}}_{\mbox{{\rm argmin}}\Phi}0)$.
From this we easily deduce that the strong convexity of $\Psi$ implies
the strong convergence of $y(\cdot)$. We recover our result by noticing
that the strong convexity of $\Psi$ is equivalent to the strong monotonicity
of its gradient, i.e., of the proximal mapping. \end{proof}

\section{Other viscosity selection principles}

Let us now examine the more general situation \begin{subequations}
\begin{eqnarray}
 &  & \upsilon\left(t\right)\in\partial\Phi\left(x\left(t\right)\right)\label{new-1}\\
 &  & \lambda\dot{x}\left(t\right)+\dot{\upsilon}\left(t\right)+\upsilon\left(t\right)+\varepsilon\left(t\right)\partial g(x\left(t\right))\ni0,\label{new-2}
\end{eqnarray}
\end{subequations} where $g$ is a convex viscosity function.\\
 Using the Minty transform, this system can be equivalently written
as \begin{subequations} 
\begin{eqnarray}
 &  & x\left(t\right)=\textrm{prox}_{\mu\Phi}\left(y\left(t\right)\right)\label{new-3}\\
 &  & \dot{y}\left(t\right)+\mu\nabla\Phi_{\mu}\left(y\left(t\right)\right)+\mu\epsilon\left(t\right)\partial g(\textrm{prox}_{\mu\Phi}\left(y\left(t\right)\right))\ni0.\label{new-4}
\end{eqnarray}
\end{subequations}
In order to recover exactly the Tikhonov approximation, we look for
some $g$ such that, for all $y\in\mathcal{H}$ 
\[
\partial(\mu g)(\textrm{prox}_{\mu\Phi}\left(y\right))=y.
\]
Equivalently 
\[
(I+\mu\partial\Phi)^{-1}=(\partial(\mu g))^{-1}.
\]
We obtain 
\[
I+\mu\partial\Phi=\partial(\mu g),
\]
that is, for all $y\in\mathcal{H}$ 
\[
\mu g(y)=\frac{1}{2}\|y\|^{2}+\mu\Phi(y).
\]
Thus, by taking $g(y)=\frac{1}{2\mu}\|y\|^{2}+\Phi(y)$, and $\Theta\left(y\right)=\mu\Phi_{\mu}\left(y\right)$,
equation (\ref{new-4}) can be equivalently written as 
\begin{equation}
\dot{y}\left(t\right)+\nabla\Theta\left(y\left(t\right)\right)+\epsilon\left(t\right)y\left(t\right)=0 .\label{eq:61}
\end{equation}
Equation (\ref{eq:61}) is a particular case of the (SDC) system (steepest
descent with control)
\[
\left(\textrm{SDC}\right)\qquad\qquad\dot{y}\left(t\right)+\partial\Theta\left(y\left(t\right)\right)+\epsilon\left(t\right)y\left(t\right)=0 .
\]
Concerning the case $\int_{0}^{+\infty}\epsilon\left(t\right)dt=+\infty$,
the first general convergence result was in \cite{Reich} (based on
previous work by \cite{Browder}), and also requires $\epsilon\left(\cdot\right)$
to be nonincreasing, and converges to zero for $t\longrightarrow+\infty$.
Under these conditions, each trajectory of (\ref{eq:61}) converges
strongly to the point of minimal norm in $C=\arg\min\Theta=\arg\min\Phi$.
In \cite{CPS}, it is proved that the convergence result still holds
without assuming $\epsilon\left(\cdot\right)$ to be nonincreasing. 

When $g(\xi)=\frac{1}{2\mu}\|\xi\|^{2}+\Phi(\xi)$, the dynamical system
(\ref{new-1})-(\ref{new-2}) becomes 
\begin{subequations} 
\begin{eqnarray}
 &  & \upsilon\left(t\right)\in\partial\Phi\left(x\left(t\right)\right)\label{new-5a}\\
 &  & \lambda\dot{x}\left(t\right)+\dot{\upsilon}\left(t\right)+\upsilon\left(t\right)+\varepsilon\left(t\right)
\left(\frac{1}{\mu} x\left(t\right) + \upsilon\left(t\right) \right) =0.\label{new-6a}
\end{eqnarray}
\end{subequations}
Equivalently
\begin{subequations} 
\begin{eqnarray}
 &  & \upsilon\left(t\right)\in\partial\Phi\left(x\left(t\right)\right)\label{new-5}\\
 &  & \dot{x}\left(t\right)+ \mu\dot{\upsilon}\left(t\right)+\mu(1+\varepsilon\left(t\right))\upsilon\left(t\right)+\varepsilon\left(t\right)x\left(t\right)=0.\label{new-6}
\end{eqnarray}
\end{subequations}

As in the preceding section, by application of the Cauchy-Lipschitz theorem (recall that $\Theta$ is differentiable, and its gradient is Lipschitz continuous), we can show that (\ref{eq:61}) admits a
strong  global solution $y\left(\cdot\right)$. Then, 
$\left(x\left(\cdot\right),\upsilon\left(\cdot\right)\right)$ with
$x\left(\cdot\right)=\textrm{prox}_{\mu\Phi}\left(y\left(\cdot\right)\right)$
and $\upsilon\left(\cdot\right)=\nabla\Phi_{\mu}\left(y\left(\cdot\right)\right)$
is a strong solution of (\ref{new-5})-(\ref{new-6}). 

Let us summarize our result in the following theorem. 

\begin{thm} Let $\Phi:\mathcal{H}\longrightarrow\mathbb{R}\cup\left\{ +\infty\right\} $
be a closed, convex, proper function, such that $C=\arg \min\Phi\neq\emptyset$.
Let us assume that 

(i) $\:$ $\underset{t\longrightarrow+\infty}{\lim}\epsilon\left(t\right)=0,$

$\vphantom{}$

(ii ) $\:$ $\int_{0}^{+\infty}\epsilon\left(t\right)dt=+\infty.$

Then, for any trajectory $\left(x\left(\cdot\right),\upsilon\left(\cdot\right)\right):\left[0,+\infty\right[\longrightarrow\mathcal{H}\times\mathcal{H}$
solution of (\ref{new-5})- (\ref{new-6}), with $y\left(t\right)=x\left(t\right)+\mu\upsilon\left(t\right)$
the following strong convergence propery holds:
   $$s-\underset{t\longrightarrow+\infty}{\lim}x\left(t\right) = s-\underset{t\longrightarrow+\infty}{\lim}y\left(t\right)=\underset{\arg\min\Phi}{\textrm{Proj}}0.$$
\end{thm} 
\begin{proof} We are in the situation examined in \cite[Theorem 2]{CPS},
which gives the strong convergence of each trajectory $y\left(\cdot\right)$
of (\ref{eq:61}) towards the point of minimal norm in $C=\arg\min\Phi$.
In order to pass from the convergence of $y(\cdot)$ to the convergence
of $x(\cdot)$ we use the relation (\ref{new-3}) 
\[
x\left(t\right)=\textrm{prox}_{\mu\Phi}\left(y\left(t\right)\right)
\]
that links the two variables. From the continuity property of the proximal mapping for the strong topology (indeed, it is a nonexpansive mapping), we deduce that each trajectory $x(\cdot)$ of (\ref{new-5})- (\ref{new-6}) converges strongly  to $\textrm{prox}_{\mu\Phi}\left(\underset{\arg\min\Phi}{\textrm{Proj}}0\right)= \underset{\arg\min\Phi}{\textrm{Proj}}0$.
To obtain this last equality, we use  the fact that the set $C$
of minimizers of $\Phi$ is invariant by the proximal mapping $\textrm{prox}_{\mu\Phi}$,
i.e., 
\[
\textrm{prox}_{\mu\Phi}(y)=y\quad{\rm for\ all\ }y\in C=\mbox{{\rm argmin}}\Phi.
\]

\end{proof}

\section{Perspective}

Let us list some interesting questions to be examined in the future: 

\begin{enumerate}
\item Examine the discrete, algorithmical version, and the corresponding
asymptotic selection property for the forward-backward algorithm. 
\item Study the case where $\lambda(t)$ depends on $t$ in an open-loop form, as in \cite{AS}.
\item Study the case where the Levenberg-Marquart regularization term is given in a closed-loop form, $\lambda(t)= \alpha (\|\dot{x} (t)\|)$  as in
\cite{ARS}.
\item Examine these questions for the  related  dynamical systems which have been considered in \cite{AA}. 
\end{enumerate}


\begin{thebibliography}{10}
\bibitem{AA} B. Abbas, H. Attouch, Dynamical systems and forward-backward
algorithms associated with the sum of a convex subdifferential and
a monotone cocoercive operator, \textit{Optimization}, (2014) http://dx.doi.org/10.1080/02331934.2014.971412.

\bibitem{AAS} B. Abbas, H. Attouch, B. F. Svaiter, Newton-like dynamics
and forward-backward methods for structured monotone inclusions in
Hilbert spaces, \textit{J. Optim. Theory Appl.}, \textbf{161} (2014),
No. 2, pp. 331-360.

\bibitem{AC} F. Alvarez, A. Cabot, Asymptotic selection of viscosity
equilibria of semilinear evolution equations by the introduction of
a slowly vanishing term, \textit{Discrete and Continuous Dynamical
Systems}, \textbf{15} (2006), pp. 921-938.

\bibitem{A} H. Attouch, Viscosity solutions of minimization problems,
\textit{SIAM J. Optimization}, \textbf{6} (1996), pp. 769-806.

\bibitem{ABM} H. Attouch, G. Buttazzo, G. Michaille, Variational
analysis in Sobolev and BV spaces. Applications to PDE's and optimization,
\textit{MPS/SIAM Series on Optimization}, \textbf{6}, Society for
Industrial and Applied Mathematics (SIAM), Philadelphia, PA, Second
edition, 2014, 793 pages.

\bibitem{AMO}H. Attouch, M.-O. Czarnecki, Asymptotic behavior of
coupled dynamical systems with multiscale aspects, \textit{Journal
of Differential Equations}, \textbf{248} (2010), pp. 1315-1344.

\bibitem{APR} H. Attouch, J. Peypouquet, P. Redont, Backward-forward
algorithms for structured monotone inclusions in Hilbert spaces, to
appear.

\bibitem{ARS} H. Attouch, P. Redont, B. F. Svaiter, Global convergence
of a closed-loop regularized Newton method for solving monotone inclusions
in Hilbert spaces, \textit{J. Optim. Theory Appl.}, 157 (2013), No.
3, pp. 624--650.

\bibitem{AS} H. Attouch, B.F. Svaiter, A continuous dynamical Newton-like
approach to solving monotone inclusions, \textit{SIAM J. Control Optim.}
\textbf{49} (2011), pp. 574--598.

\bibitem{BC} J. B. Baillon, R. Cominetti, A convergence result for
nonautonomous subgradient evolution equations and its application
to the steepest descent exponential penalty trajectory in linear programming,
\textit{Journal of Functional Analysis}, \textbf{187} (2001), pp.~263--273.

\bibitem{BaCom} H.H. Bauschke, P.L. Combettes, Convex analysis and
monotone operator theory in Hilbert spaces. Springer, New York (2011).

\bibitem{Br} H. Brézis, Opérateurs maximaux monotones et semi-groupes
de contractions dans les espaces de Hilbert, North-Holland/Elsevier,
New-York, 1973.

\bibitem{Browder}F.E. Browder, Nonlinear Operators and Nonlinear
Equations of Evolution in Banach Spaces, Proc. Sympos. Pure Math.,
vol. 18 (part 2), Amer. Math. Soc., Providence, RI, 1976.

\bibitem{Cabot} A. Cabot, The steepest descent dynamical system with
control. Applications to constrained minimization, \textit{ESAIM:
Control, Optimization and Calculus of Variations}, \textbf{10} (2004),
pp. 243-258.

\bibitem{CPS} R. Cominetti, J. Peypouquet, S. Sorin, Strong asymptotic
convergence of evolution equations governed by maximal monotone operators
with Tikhonov regularization, J. Differential Equations, \textbf{245}
(2008), pp.~3753--3763.

\bibitem{FurMiyKen} H. Furuya, K. Miyashiba, N. Kenmochi, Asymptotic
behavior of solutions to a class of nonlinear evolution equations,
{\em Journal of Differential Equations}, \textbf{62} (1986), pp.
73--94.

\bibitem{Ha} A. Haraux, Systèmes dynamiques dissipatifs et applications.
RMA 17, Masson, Paris, (1991).

\bibitem{Hir} S. Hirstoaga, Approximation et résolution de problèmes
d'équilibre, de point fixe et d'inclusion monotone, PhD thesis, UPMC
Paris 6, (2006).

\bibitem{Reich} S. Reich, Nonlinear evolution equations and nonlinear
ergodic theorems, Nonlinear Anal. 1 (1976) 319\textendash 330.

\bibitem{RW} R.T Rockafellar, R. J-B. Wets, Variational Analysis,
Grundlehren der mathematischen Wissenschaften \textbf{317}, Springer-Verlag,
(1998).

\bibitem{Son} E.D. Sontag, Mathematical Control Theory, second edition.
Springer-Verlag, New-York (1998).\end{thebibliography}
\end{document}